\RequirePackage{fix-cm}
\documentclass[smallextended]{svjour3}       
\smartqed  
\usepackage{graphicx}
\usepackage{amsmath,amssymb}
\begin{document}

\title{CONTINUOUS COMODULES
}


\author{Nikken Prima Puspita         \and
        Indah Emilia Wijayanti  \and
        Budi Surodjo
}


\institute{Nikken Prima Puspita \at
              Department of Mathematics, Universitas Diponegoro Semarang-Indonesia.  \\
              \email{nikkenprima@lecturer.undip.ac.id}           
           \and
           Indah Emilia Wijayanti \at
              Department of Mathematics, Universitas Gadjah Mada, Yogyakarta-Indonesia.
           \and
           Budi Surodjo \at
              Department of Mathematics, Universitas Gadjah Mada, Yogyakarta-Indonesia.
}

\date{Received: date / Accepted: date}

\maketitle

\begin{abstract}
Let $R$ be a commutative ring with unity and $C$ be an $R$-coalgebra. The ring $R$ is clean if every $ r\in R $ is the sum of a unit and an idempotent element of $R$. An $R$-module $M$ is clean if the endomorphism ring of $M$ over $R$ is clean. Moreover, every continuous module is clean. We modify this idea to the comodule and coalgebra cases. A $C$-comodule $M$ is called a clean comodule if the $C$-comodule endomorphisms of $M$ are clean. We introduced continuous comodules and proved that every continuous comodules is a clean comodule.
\end{abstract}

\section{Introduction}
\label{intro}
Let $R$ be a commutative ring with unity. The notions of clean ring and clean module motivate us to bring this concept to the coalgebra and comodule case. In 1977, Nicholson, W. K. defined clean rings \cite{nic77}. The ring $ R $ is said to be clean if every element of $ R $ can be expressed as the sum of a unit and an idempotent element. Furthermore, clean rings are a subclass of exchange rings \cite{warfield,craw}.

We denote the endomorphisms of $R$-module $R$ by $End_{R}(R)$, which is isomorphic to the ring $R$. Consequently, $R$ is clean if and only if the ring $End_{R}(R)$ is also clean. In general, for any $R$-module $M$, several authors have studied the cleanness of the endomorphism of $R$-modules $M$ (denoted by $End_{R}(M))$. In 1998, Nicholson, W.K. and Varadarajan, K. \cite{nic98} showed that the ring of a linear transformation of a countable linear vector space is clean. The same result is valid for arbitrary vector spaces over a field and any vector space over a division ring (\cite{sear} and \cite{nic04}).

In 2006, Camillo et al. \cite{cam06} introduced the clean module. A clean module is an $R$-module $M$, in which the ring of $End_{R}(M)$ is a clean ring. We recall results in \cite{cam06} and \cite{cam12} regarding the necessary and sufficient conditions of clean modules (see Proposition 2.2 and 2.3 in \cite{cam06}). In \cite{cam06}, the authors also give some examples of a module with a clean property and prove that the endomorphism ring of a continuous module is clean.

In 1969, Sweedler proposed a coalgebra over a field. Later, this ground field was generalized to any ring \cite{wisbauer}. On the other hand, a comodule over a coalgebra is well-known to be a dualization module over a ring. Throughout, $(C,\Delta,\varepsilon)$ is a coassociative and counital $R$-coalgebra and $(M,\varrho^{M})$ is a right $C$-comodule. Moreover, we consider $C^{\ast}=Hom_{R}(C,R)$. For arbitrary $f,g\in C^{\ast}$, the convolution product in $C^{\ast}$ (denoted by $"\ast"$) defined as below :
\begin{equation}\label{convo}
f\ast g =\mu\circ(f\otimes g)\circ \Delta
\end{equation}
The set $C^{\ast}$ is an $R$-algebra (ring) over addition and convolution products as in Equation (\ref{convo}) (see \cite{wisbauer}). We refer to $(C^{\ast},+,\ast)$ as a dual $R$-algebra of $C$. The structure of $C^{\ast}$ implies some relationship between the category of $C$-comodules and $C^{\ast}$-modules. Any right (left) $C$-comodule $M$ is a left (right) module over the dual algebra $C^{\ast}$ by a scalar multiplication as in the following equation:
\begin{equation}\label{scalarmul}
  \rightharpoonup : C^{\ast}\otimes_{R} M\rightarrow M, g\otimes m \mapsto (I_{M}\otimes_{R}g)\circ \varrho^{M}(m),
\end{equation}
which is $g\otimes m \mapsto (I_{M}\otimes_{R}g)\circ \varrho^{M}(m)=\sum m_{\underline{0}}g(m_{\underline{1}})\in M$ \cite{wisbauer}. Furthermore, the category of right $C$-comodule ($\mathbf{M}^{C}$) is a subcategory of the category of left $C^{\ast}$-module ($_{C^{\ast}}\mathbf{M}$).

In this paper, we assumed that $C$ does not always satisfy the $\alpha$-condition. However, $End^{C}(M)$ is a subring of $_{C^{\ast}}End(M)$. Although $M$ is a clean $C^{\ast}$-module, it does not imply that the ring of $End^{C}(M)$ is clean. For example, $\mathbb{Z}$ is a subring of $\mathbb{R}$. Hence, $\mathbb{R}$ is a clean ring, but $\mathbb{Z}$ is not clean.  Based on this fact, as a dualization concept of clean modules, clean comodules are defined in the following way.
\begin{definition}\label{cleancomodule}
Let $(C,\Delta,\varepsilon)$ be an $R$-coalgebra. A right (left) $C$-comodule $M$ is said to be a clean comodule if the ring $End^{C}(M)$ is clean.
\end{definition}
Definition \ref{cleancomodule} means that $M\in \mathbf{M^{C}}$ is clean if for any $f\in End^{C}(M), f=u+e$ where $u$ is a unit and $e$ is an idempotent in $End^{C}(M)$. Moreover, since any $R$-coalgebra $C$ is a comodule over itself, $C$ is clean coalgebra over $R$ if the ring of $End^{C}(C)$ is clean. In section 2, we start our investigation by giving propositions to explain the necessary and sufficient conditions of clean comodules.

We study continuous and quasi-continuous modules in \cite{muler}. In 2006, Camillo et al. \cite{cam06} proved that every quasi-continuous module is clean. By using this fact, we bring the concept of continuous modules in \cite{cam06} to define continuous and quasi-continuous in comodule structures. Moreover, in Section 3, as our main result, we prove that a continuous (quasi-continuous) comodule is clean.

\section{The Necessary and Sufficient Condition of Clean Comodules}
\label{sec:1}
We recall Lemma 2.1 and Proposition 2.2 in \cite{cam06} as a necessary and sufficient condition of clean modules. We will make some modifications for clean comodules by changing the endomorphisms of $R$-modules to $C$-comodules. Moreover, we need to make sure that the set of kernel is also a subcomodule. Our goal is to give some necessary and sufficient conditions for clean comodules.
\begin{lemma}\label{NSCCleanCom1}
Let $C$ be a flat $R$-module and be an $R$-coalgebra and $M$ be a right $C$-comodule. Consider $S=End^{C}(M)$ and $f,e\in End^{C}(M)$, where $e$ is an idempotent, $A=Ker(e)$ and $B=Im(e)$. Then $f-e$ is a unit in $End^{C}(M)$ if and only if there exists a $C$-comodule decomposition $M=X\oplus Y$ such that $f(A)\subseteq X, 1-f(B)\subseteq Y$, and both $f:A\rightarrow X$ and $1-f:B\rightarrow Y$ are isomorphisms.
\end{lemma}
\begin{proof}
Let $S=End^{C}(M)\subseteq End_{R}(M)$ and $f,e\in S$, where $e$ is an idempotent and $A=Ker(e)$ and $B=Im(e)$.

$\Rightarrow$
Suppose that $f-e=u$ or $f=u+e$ for some unit $u\in S$ (this means $f$ is clean in $S$). We prove that there exists a $C$-comodule decomposition $M=X\oplus Y$ such that $f(A)\subseteq X,(1-f)(B)\subseteq Y$, and both $f:A\rightarrow X$ and $1-f:B\rightarrow Y$ are $C$-comodule isomorphisms .

Consider $e\in End^{C}(M)$ is an idempotent and $1=I_{M}$, where $A=Ker(e)$ and $B=Im(e)$. Since $End^{C}(M)$ is a subring of $End_{R}(M)$, $e\in End_{R}(M)$. As an idempotent $R$-module endomorphism and by using \cite{wis91} (see page 58)\\
\begin{equation}
Ker(e)=Im(1-e)
\end{equation}
\begin{equation}
Im(e)=Ker(1-e).
\end{equation}
and,
\begin{equation}\label{moduledecomp}
M = Im(1-e)\oplus Im(e)=Ker(e)\oplus Im(e)=A\oplus B.
\end{equation}
The decomposition in Equation \ref{moduledecomp} is an $R$-module decomposition. Since $C$ is flat, $e$ is a $C$-pure morphism; moreover, $A=Ker(e)$ is a $C$-subcomodule of $M$.
We now need to prove that $1-e$ is a $C$-comodule morphism. The commutativity of $C$-comodule morphism of $e$ means $\varrho^{M}\circ e= (e\otimes I_{C}) \circ \varrho^{M}$. Moreover, for every $m\in M$, we have:
\begin{align*}
\varrho^{M}\circ (1-e)(m)&=\varrho^{M}\circ (1(m)-e(m))\\
    &=(\varrho^{M}\circ 1)(m)-(\varrho^{M}\circ e)(m)\\
    &=((1\otimes I_{C}) \circ \varrho^{M})(m)-((e\otimes I_{C}) \circ \varrho^{M})(m)\\
    & \text{( since $1$ and $e$ are $C$-comodule morphism)}\\
    &=(((1\otimes I_{C})-(e\otimes I_{C})) \circ \varrho^{M})(m)\\
    &=(((1-e)\otimes I_{C})\circ \varrho^{M})(m).
    \end{align*}
Hence, $\varrho^{M}\circ (1-e)= ((1-e)\otimes I_{C})\circ \varrho^{M}$ or $(1-e)$ is a $C$-comodule morphism. Consequently, $Ker(1-e)$ is a $C$-subcomodule, which implies that
\begin{center}
$M=Ker(e)\oplus Ker(1-e)=Ker(e)\oplus Im(e)=A\oplus B$
\end{center}
is a $C$-comodule decomposition.
Let $f:M\rightarrow M$ be a $C$-comodule morphism and $A=Ker(e)$ and $B=Im(e)$. Suppose that $f=u+e\in S$ and put $X=u(A)$ and $Y=u(B)$.
\begin{enumerate}
\item For $f|_{A}$, we have
\begin{align*}
f(A)&=(e+u)(A)\\
&=e(A)+u(A)\\
&=0+u(A), \text{ (since $A=Ker(e)$)}\\
&=X.
\end{align*}
Hence, $f(A)\subseteq X$.
\item Furthermore, for $1-f|_{B}$,
\begin{align*}
1-f(B)&=B-f(B)\\
&=B-(u+e)(B)\\
&=B-u(B)-e(B)\\
&=(1-e)(B)-u(B), \text{ (since $B=Im(e)=Ker(1-e))$}\\
&=-u(B)\\
&=-Y
\end{align*}
This means $1-f(B)\subseteq -Y\subseteq Y$.
\item We are going to prove that $f:A\rightarrow X$ is a $C$-comodule isomorphism. Since $f=u+e$,
\begin{align*}
f(1-e)&=(e+u)(1-e)\\
&=(e-e+u-ue)\\
&=u(1-e).\\
\end{align*}
By $(1-e)(M)=Im(1-e)=Ker(e)$, then
\begin{align*}
f((1-e)(M))&=u((1-e)(M))\\
\Leftrightarrow f(Ker(e))&= u(Ker(e))\\
\Leftrightarrow f(A)&=u(A)\\
\Leftrightarrow&=X.
\end{align*}
This means $f_{A}\simeq X$ or $f:A\rightarrow X$ is an isomorphism. In particular, $f:A\rightarrow X$ is a $C$-comodule isomorphism.
\item We will now prove that $1-f$ is a $C$-comodule isomorphism. We have already proven that if $f$ is a $C$-comodule morphism, then so is $1-f$. Moreover, we have:
\begin{align*}
(1-f)e&=e-fe\\
&=ee-fe\\
&=(e-f)e\\
&=(e-(e+u))e\\
&=-ue.
\end{align*}
Since $e(M)=Im(e)=B$, then
\begin{align*}
&((1-f)e)(M)=-ue(M)\\
&(1-f)(e(M))=-u(e(M))\\
&\Leftrightarrow (1-f)(B)=-u(B)=-Y\subseteq Y
\end{align*}
Since $u$ is a unit, $(1-f)|_{B}$ is an isomorphism. In particular, $1-f:B\rightarrow Y$ is an isomorphism.
\item Consequently, since $f:A\rightarrow X$ and $1-f:B\rightarrow Y$ are $C$-comodule isomorphisms we have a decomposition.
\begin{center}
$M=A\oplus B\simeq X\oplus Y$.
\end{center}
\end{enumerate}

$(\Leftarrow)$
Conversely, we prove that $u:=f-e$ is a unit in $S$, as follows:
\begin{align*}
u(M)&=u(A+B)\\
&=(f-e)(A+B)\\
&=f(A)-e(A)-f(B)+e(B)\\
&=f(A)-0-f(B)+((f-u)(B)),\text{ (since $A=Ker(e)$ and $f=e+u$)}\\
&=f(A)-u(B)\\
&=f(A)+(1-f)(B), \text{ (since $(-u)(B)=(1-f)(B)$)}\\
\end{align*}
Hence, $f:A\rightarrow X$ and $1-f:B\rightarrow Y$ are isomorphisms. Therefore,
\begin{align*}
u(M)&=f(A)+(1-f)(B)\\
&=X+Y\\
&=I_{M}(M).
\end{align*}
Clearly, $u=I_{M}$ is a unit in $S=End^{C}(M)$. In particular, $f-e=u$ is a unit in $S$.
\end{proof}

From Lemma \ref{NSCCleanCom1}, we get the following proposition, which is useful to prove our Main Theorem.
\begin{proposition}\label{NSCCleanCom2}
Let $C$ be a flat $R$-module and an $R$-coalgebra, and let  $M$ be a right $C$-comodule. An element $f\in End^{C}(M)$ is clean if and only if there exist $C$-comodule decompositions $M=A\oplus B=X\oplus Y$ such that $f(A)\subseteq C, (1-f)(B)\subseteq Y$, and both $f:A\rightarrow X$ and $1-f:B\rightarrow Y$ are isomorphisms.
\end{proposition}
\begin{proof}
($\Leftarrow$ ) Based on Lemma \ref{NSCCleanCom1} we have $M=A\oplus B$ as a $C$-comodule decomposition. Take $X=u(A)$ and $Y=u(B)$, then $f|_{A}:A\rightarrow X$ and $(1-f)|_{B}:B\rightarrow Y$ are isomorphisms. In particular, there exists the $C$-comodule decomposition $M=A\oplus B\simeq X\oplus Y$.

($\Rightarrow$) Conversely, consider the decomposition
\begin{center}
$M=A\oplus B$.
\end{center}
There exist the projection map
\begin{equation}\label{proj}
\pi_{B}:M\rightarrow B, a+b\mapsto b
\end{equation}
and the inclusion map
\begin{equation}\label{inclus}
\iota_{B}:B\rightarrow M,
\end{equation}
such that $f_{B}=\iota_{B}\circ \pi_{B}$ is an idempotent in $End_{R}(M)$. Put
\begin{center}
$Y=Im(f_{B})$ and $X=Ker(f_{B})$
\end{center}
and $X=A$ and $Y=B$.
Now we need to check that $\iota_{B}\circ \pi_{B}:M\rightarrow M$ in $End^{C}(M)$. It is clear that the inclusion map $\iota_{B}$ is a $C$-comodule morphism. For every $m=a+b\in M$, we have:
\begin{align*}
(\pi_{B}\otimes I_{C})\circ \varrho^{M}(m)&=(\pi_{B}\otimes I_{C})\circ \varrho^{M}(a+b)\\
&=(\pi_{B}\otimes I_{C})(\varrho^{M}(a))+(\pi_{B}\otimes I_{C})(\varrho^{M}(b))\\
&=(\pi_{B}\otimes I_{C})(\sum a_{0}\otimes a_{1})+(\pi_{B}\otimes I_{C})(\sum b_{0}\otimes b_{1})\\
&=(\sum \pi_{B}(a_{0})\otimes I_{C}(a_{1}))+(\sum \pi_{B}(b_{0})\otimes I_{C}(b_{1}))\\
&=0+(\sum \pi_{B}(b_{0})\otimes I_{C}(b_{1}))\text{ (since $a_{0}\in A$)}\\
&=\sum b_{0}\otimes b_{1}\\
&=\varrho^{M}|_{B}(b)\text{ (since $B$ is a $C$-subcomodule of $M$)}\\
&=\pi_{B}\circ \varrho^{M}|_{B}(m).
\end{align*}
Thus, $\pi_{B}$ is a $C$-comodule morphism; moreover, $e=\iota_{B}\circ \pi_{B}$ is an idempotent in $End^{C}(M)$.\\
Now, we have that there exist $C$-comodule decomposition $M=X\oplus Y$ such that $f(A)\subseteq X$ and $(1-f)(B)\subseteq Y$, and $f:A\rightarrow X$ and $1-f:B\rightarrow Y$ are isomorphisms. As the implication of Lemma \ref{NSCCleanCom1}, $f$ is a clean element in $End^{C}(M)$ or $M$ is a clean $C$-comodule.
\end{proof}
Based on Lemma \ref{NSCCleanCom1} and Proposition \ref{NSCCleanCom2}, we have already obtained the necessary and sufficient condition of clean $C$-comodules.

\section{Essential Subcomodules and Continuous Comodules}
Before studying the continuous comodule, we need to understand some preliminary structures, which are motivated by a similar situation in module theory. We define an essential subcomodule below:
\begin{definition}\label{essubcom}
Let $M$ be a comodule over $R$-coalgebra $C$. A $C$-subcomodule of $N$ is called an essential subcomodule of $M$ (or $M$ is essential extension of $N$) if for any nonzero $C$-subcomodule $L\subseteq M$, we have $L\cap N\neq \{0\}$. A $C$-subcomodule $N$ is called \textbf{closed} if it does not have any proper essential extension in $M$. If $N'$ is a closed subcomodule of $M$ and $N$ is an essential subcomodule of $N'$, then we call $N'${\em a closure of $N$ in $M$}.
\end{definition}
Definition \ref{essubcom} means that for every nonzero $C$-subcomodule $L\subseteq M$ with $L\cap N=\{0\}$, $L=\{0\}$. Throughout, we denote an essential $C$-subcomodule $N$ of $M$ by $N\subseteq ^{e}M$.

In module theory we have certain properties related to the essential $R$-submodule. Let $K,L$ be submodules of $M$ with $K\subseteq L\subseteq M$. Therefore, $K\subseteq_{e}M$ if and only if $K\subseteq_{e} L$ and $L\subseteq_{e} M$.
Furthermore, we apply the property of essential submodules to comodule structure.
\begin{lemma}\label{lemmaessen}
Let $M$ be a comodule over an $R$-coalgebra $R$ and $K,N$ be $C$-subcomodules of $M$ where $K\subseteq N\subseteq M$. Then  $K\subseteq^{e}M$ if and only if $K\subseteq^{e} N$ and $N\subseteq^{e} M$
\end{lemma}
\begin{proof}
($\Rightarrow$) Let $K\subseteq^{e}M$. This means that for any nonzero $C$-subcomodule $H\subseteq M$ we have $K\cap H\neq \{0\}$. Suppose that $H'$ is a non-zero subcomodule of $N$. Since $H'$ is also a subcomodule of $M$ and $K\subseteq^{e}M$,  $K\cap H'\neq \{0\}$. Thus, $K\subseteq^{e} N$. Moreover, for any nonzero $C$-subcomodule $H\subseteq M$, since $K\subseteq^{e}M$, we have $H\cap K\neq \{0\}$. Hence, $\{0\}\neq H\cap K\subseteq K\subseteq N$. Thus, $H\cap N\neq \{0\}$ or $N\subseteq^{e}M$.

($\Leftarrow$) Now suppose that $K\subseteq^{e} N$ and $N\subseteq^{e} M$. We prove that $K\subseteq^{e}M$. Given an arbitrary $C$-subcomodule $L\subseteq M$ with $K\cap L=\{0\}$. Since $K\subseteq^{e} N$ and $K\cap L=\{0\}$, $L\cap N=\{0\}$. On the other hand, $N\subseteq^{e} M$ and $L$ is a $C$-subcomodule of $M$, which implies $L=0$. In particular, $K\subseteq^{e}M$.
\end{proof}

We have some modifications of Remark 19.4 in \cite{tugan} as follows:
\begin{lemma} \label{clossubcom} Let $M$ be a $C$-comodule. The following assertions hold:
\begin{description}
  \item[(1)] Every $C$-subcomodule of $M$ has at least one closure in $M$.
  \item[(2)] If $G$ and $N$ are two $C$-subcomodules of $M$ where $G\cap N=0$, then $G$ has at least one closed complement $H$ in $M$ that contains $N$.
  \item[(3)] If $N_{1}$ and $N_{2}$ are two subcomodules of $M$ such that $M_{1}\subseteq N_{1}$, $M_{2}\subseteq N_{2}$, $M_{1}\cap M_{2}=0$, then $M$ is an essential extension of $M_{1}\oplus M_{2}$, and $M_{1}\cap K\neq 0$ for every subcomodule $K$ of $M$ that properly contains $M_{2}$.
\end{description}
\end{lemma}
\begin{proof}
\begin{description}
  \item[(1)] Let $N$ be a $C$-subcomodule of $M$ and
   \begin{center}
   $\xi=\{H_{i}| H_{i}$ be a $C$-subcomodule of $M$, $N\subseteq^{e}H_{i}\}$.
    \end{center}
 The set $\xi$ is not empty and contains the union of any ascending chain of its elements. By Zorn's Lemma, there exists a maximal element of $\xi$, i.e., $H=\cup_{i\in I}H_{i}\in \xi$ such that $N\subseteq^{e}H$ and $H$ does not have a proper essential extension in $M$. Therefore, $H$ is a closure of $N$ in $M$.
  \item[(2)] Let
   \begin{center}
   $\xi'=\{H_{i}| H_{i}$ be a subcomodule of $M$, $N\subseteq H_{i}$ and $G\cap H_{i}=\{0\}$ for any $i\}$.
\end{center}
The  set $\xi'$ is not empty, since $N\in \xi'$. Moreover, the set $\xi'$ contains the union of any ascending chain of its elements.
      \begin{enumerate}
      \item Using Zorn's Lemma, $\xi'$ contains a maximal element $H$ such that $N\subseteq H,G\cap H=\{0\}$ and $G\oplus H\subseteq M$.
      \item We want to show $G\oplus H\subseteq^{e}M$. Let $X$ be a $C$-subcomodule of $M$, then
      \begin{enumerate}
      \item \textbf{Case 1} : If $X\in \xi'$, then $N\subseteq X\subseteq H$, since $H$ is the maximal element. Thus, $X\cap (G+H)=X\cap G+ X\cap H\neq\{0\}$.
      \item \textbf{Case 2}: If $X$ is not in $\xi'$, then $G\cap X\neq \{0\}$. Consequently, $X\cap (G+H)\neq\{0\}$.
      \end{enumerate}
      Based on both cases, we have that $G\oplus H\subseteq^{e}M$.
      \item For any $K$ that is a $C$-subcomodule of $M$ that properly contains $H$, it is clear that $G\cap K\neq \{0\}$, since $H$ is an element maximal in $\xi'$.
      \end{enumerate}
      Assume that $H$ is not closed, i.e., $K$ is a subcomodule of $M$ such that $H\subseteq^{e}K$ and $K$ properly contains $H$. Hence, $G\cap K\neq \{0\}$.

Moreover, since $H\subseteq^{e}K$ and $G\cap K$ is a $C$-subcomodule of $K$, we have $(G\cap K)\cap H=G\cap (K\cap H)=G\cap H\neq \{0\}$. This is a contradiction to $G\cap H=\{0\}$. Consequently, $H$ is closed and $N\subseteq H$.
  \item[(3)] Let $N_{1}$ and $N_{2}$ be two subcomodules of $M$.
  \begin{enumerate}
  \item By Lemma \ref{clossubcom}, there is a closed subcomodule $M_{2}$ of $M$ such that $M_{2}\subseteq N_{2}$ and $N_{1}\cap M_{2}=\{0\}$. Hence, $M$ is an essential extension of $N_{1}\oplus M_{2}$ and $M_{1}\cap K\neq \{0\}$ for every $C$-subcomodule $K$ of $M$ properly containing $M_{2}$.
  \item Using (1), $N_{1}$ has at least one closure $M_{1}$ in $M$. Hence, $N_{1}\subseteq^{e} M_{1}$ and $N_{1}\cap M_{2}=\{0\}$. Since $M_{1}\cap M_{2}$ is a $C$-subcomodule of $M_{1}$ and $M_{1}\cap M_{2}\cap N_{1}=\{0\}$, $M_{1}\cap M_{2}=\{0\}$.
  \item Moreover, we already know that $N_{1}\oplus M_{2}\subseteq^{e}M$ and $N_{1}\oplus M_{2}\subseteq M_{1}\oplus M_{2}\subseteq M$. By Lemma \ref{lemmaessen}, we have $M_{1}\oplus M_{2}\subseteq^{e} M$.
  \end{enumerate}
\end{description}
\end{proof}

Now we bring the concept of continuous modules to the category of right $C$-comodules $\mathbf{M^{C}}$ and give the following notions:
\begin{description}
  \item[$(CM_{1})$] For every subcomodule $A\in M$, there exists a direct summand $K$ of $M$ such that $A\subseteq^{e}K$.
  \item[$(CM_{2})$] If a subcomodule $A$ of $M$ is isomorphic to a summand of $C$-subcomodule $M$, then $A$ is a summand of $M$.
  \item[$(CM_{3})$] If $N_{1}$ and $N_{2}$ are summands of $C$-comodule $M$ such that $N_{1}\cap N_{2}=\{0\}$, then $N_{1}\oplus N_{2}$ is a summand of $M$.
\end{description}

We use the above statements for defining continuous and quasi-continuous comodules.
\begin{definition}\label{contcom}
A comodule $M$ over $R$-coalgebra $C$ is called CS if it satisfies $(CM_{1})$, and $M$ is called a continuous $C$-comodule if it satisfies $(CM_{1})$ and $(CM_{2})$; $M$ is called a quasi-continuous $C$-comodule if it satisfies $(CM_{1})$ and $(CM_{3})$.
\end{definition}

The quasi-continuous module has special characteristics related to the idempotent element. The following lemma is a consequence of Lemma \ref{clossubcom}.
\begin{lemma}\label{idemcom}
Let $C$ be a flat $R$-module and a coalgebra over $R$. If $M$ is a quasi-continuous $C$-comodule, then every idempotent endomorphism of any $C$-subcomodule $M$ can be extended to an idempotent endomorphism of $M$.
\end{lemma}
\begin{proof}
Let $N$ be a $C$-subcomodule of a quasi-continuous module $M$ and let $e\in End^{C}(N)$ with $C$ be flat. Then by using the same argument as an Lemma \ref{NSCCleanCom1}, we have
\begin{equation*}
N=Im(e)\oplus Im(I_{N}-e)
\end{equation*}
is a $C$-subcomodule decomposition. Put $N_{1}=Im(e)$ and $N_{2}=Im(I_{N}-e)$ such that $N=N_{1}\oplus N_{2}$. By Lemma \ref{clossubcom}, there exist closed subcomodules $Q_{1},Q_{2}\subseteq M$ such that $N_{1}\subseteq Q_{1}$ and $N_{2}\subseteq Q_{2}, Q_{1}\cap Q_{2}=\{0\}$ and $Q_{1}\oplus Q_{2}\subseteq^{e}M$. Since $M$ is quasi-continuous, we have the following:
\begin{enumerate}
\item based on $(CM_{1})$, there are direct summands $K_{1}, K_{2}$ of $M$ such that $Q_{1}\subseteq^{e}K_{1}$ and $Q_{2}\subseteq^{e}K_{2}$;
\item since $Q_{1}, Q_{2}$ are closed,  $Q_{1}=K_{1}$ and $Q_{2}=K_{2}$ are direct summands of $M$;
\item based on $(CM_{3})$, $Q_{1}\oplus Q_{2}$ is also a summand of $M$ or $M=Q_{1}\oplus Q_{2}\oplus Q_{3}$ for a $C$-subcomodule $Q_{3}$ of $M$.
\end{enumerate}
Take $M=Q_{1}\oplus Q_{2}\oplus Q_{3}$ for a $C$-subcomodule $Q_{3}\subseteq M$. Therefore, there is a projection map
\begin{center}
$\pi_{Q_{1}}:M\rightarrow Q_{1}$
\end{center}
where the kernel of $\pi_{Q_{1}}$ is $Q_{2}\oplus Q_{3}$ and the inclusion $\iota_{Q_{1}}:Q_{1}\rightarrow M$ is a $C$-comodule morphism.

Now we construct the map
\begin{center}
$e'=\iota_{Q_{1}}\circ \pi_{Q_{1}}:M\rightarrow Q_{1}\rightarrow M, m\mapsto \iota_{Q_{1}}\circ\pi_{Q_{1}}(m)$.
\end{center}
\begin{enumerate}
\item The map $e'$ is an idempotent, i.e., for any $m=q_{1}+q_{2}+q_{3}\in M$, where $q_{i}\in Q_{i}$ for $i=1,2,3$, we have:
\begin{align*}
e'\circ e'(m)&=e'(e'(q_{1}+q_{2}+q_{3}))\\
&=e'(\pi_{Q_{1}}(q_{1}+q_{2}+q_{3}))\\
&=e'(q_{1})\\
&=e'(q_{1}+0+0)\\
&=q_{1}
\end{align*}
and $e'(m)=e'(q_{1}+q_{2}+q_{3})=q_{1}$.
\item We prove that $e'$ is a $C$-comodule morphism, i.e., $\varrho^{M}\circ e'=(e'\otimes I_{C})\circ \varrho^{M}$. Take any $m=q_{1}+q_{2}+q_{3}\in M$ where $q_{i}\in Q_{i}$ for $i=1,2,3$, then we have
    \begin{align*}
    \varrho^{M}\circ e'(m)&=\varrho^{M}\circ e' (q_{1}+q_{2}+q_{3})\\
    &=\varrho^{M}\circ \iota_{Q_{1}}\circ \pi_{Q_{1}}(q_{1}+q_{2}+q_{3})\\
    &=\varrho^{M}(q_{1})\\
    &=\sum q_{1\underline{0}}\otimes q_{1\underline{1}}
    \end{align*}
    and
    \begin{align*}
    (e'\otimes I_{C})\circ \varrho^{M}(m)&=(e'\otimes I_{C})\circ \varrho^{M}(q_{1}+q_{2}+q_{3})\\
    &=(e'\otimes I_{C})\circ(\varrho^{M}(q_{1})+\varrho^{M}(q_{2}+q_{3}))\\
    &=((e'\otimes I_{C})\circ\varrho^{M})(q_{1})+((e'\otimes I_{C})\circ\varrho^{M})(q_{2}+q_{3})\\
    &=(e'\otimes I_{C})\circ\varrho^{M}(q_{1})+0\\
    &\text{( since $Q_{2}+Q{3}$ is a $C$-subcomodule of $M$ and also-)}\\
    &\text{( kernel of $\pi_{Q_{1}}$)}\\
    &=(\iota_{Q_{1}}\circ\pi_{Q_{1}}\otimes I_{C})(\sum q_{1\underline{0}}\otimes q_{1\underline{1}})\\
    &=\sum q_{1\underline{0}}\otimes q_{1\underline{1}} \text{( since $Q_{1}$ is a $C$-suncomodule of $M$)}
    \end{align*}
\end{enumerate}
Hence, $e'=\iota_{Q_{1}}\circ \pi_{Q_{1}}$ is a $C$-comodule endomorphism of $M$. In particular, any idempotent $e\in End^{C}(N)$ can be extended to the idempotent $C$-comodule morphism of $M$, i.e., $e'$.
\end{proof}

In this section, we have some valuable result modifications of \cite{cam06}, which can be used in the comodule situation.

\section{Clean Continuous Comodules}
Following a similar result in module theory, i.e., any continuous modules are clean \cite{cam06}, we investigate whether any continuous comodules are clean. We begin this section by proving some preliminary lemmas.

Related to clean modules, some researchers have investigated cleanness of the endomorphisms of vector space (over a field or a division ring), for example, \cite{sear}. On the first lemma, we show that the endomorphisms of a $C$-comodule that is a direct sum of $n$ $C$-comodules are clean.
\begin{lemma} Let $M$ be a comodule over a coalgebra $C$ and $L=\oplus_{n\geq0}M_{n}$, where $M=M_{n}$ for each $n$. For any $m\in M$, write $m_{n}$ for the element $m$ lying in $M_{n}=M$ and define the (forward) "shift operator" $f$ on $L$ by $f(m_{n})=m_{n+1}$ for all $n$. Then $f$ is clean in $End^{C}(L)$
\end{lemma}
\begin{proof}
Suppose that $L=\oplus_{n\geq0}M_{n}$ is a $C$-comodule, since the coproduct of $C$-comodules is a $C$-comodule \cite{wisbauer}. Take any $f\in End^{C}(L)\subseteq End_{R}(L)$, i.e.,
\begin{center}
$f:L\rightarrow L$ or $f:\oplus_{n\geq0}M_{n}\rightarrow \oplus_{n\geq0}M_{n}$
\end{center}
where $f(\sum_{n\geq0}m_{n})=\sum_{n\geq0}m_{n+1}$.
In modules theory, \cite{sear} has already proved this lemma by choosing $R$-module endomorphism of $\oplus_{n\geq0}M_{n}$ such that $f=u+e$ for a unit $u$ and an idempotent $e$ in $End_{R}(L)$. By using a similar argument, we prove that the following $e$ is an $C$-comodule endomorphism.
\begin{center}
$e:\oplus_{n\geq0}M_{n}\rightarrow \oplus_{n\geq0}M_{n}$
\end{center}
in $End_{R}(L)$ defined as $e(\sum_{n\geq0}m_{n})=\sum_{n\geq0}e(m_{n})$ where
\begin{center}
  $e(m_{2n})=m_{2n}$ and $e(m_{2n+1})= m_{2n+2}-m_{2n}$, for all $n\geq 0$.
\end{center}
Clearly, $e$ is an idempotent element in $End_{R}(L)$. Moreover, we want to prove $e\in End^{C}(L)$. For any $\sum_{n\geq0}m_{n}\in L$,
\begin{align*}
\varrho^{L}\circ e(\sum_{n\geq0}m_{n})&=\varrho^{L}(m_{0},m_{4}-m_{2},m_{2},m_{8}-m_{6},m_{4},...,m_{2n},m_{2n+2}-m_{2n},....)\\
&=(\varrho^{M}(m_{0}),\varrho^{M}(m_{4}-m_{2}),\varrho^{M}(m_{2}),\varrho^{M}(m_{8}-m_{6}),\varrho^{M}(m_{4}),...,\\
&\hspace{0,5cm}\varrho^{M}(m_{2n}),\varrho^{M}(m_{2n+2}-m_{2n}),....)\\
&=(m_{0\underline{0}}\otimes m_{0\underline{1}},.......,m_{2n\underline{0}}\otimes m_{2n\underline{1}},m_{(2n+2)\underline{0}}-m_{2n\underline{0}}\otimes m_{(2n+2)\underline{1}}\\
&\hspace{0,5cm}-m_{2n\underline{1}},...)\\
&=((e\otimes I_{C})(m_{0\underline{0}}\otimes m_{0\underline{1}}),.......,(e\otimes I_{C})(m_{2n\underline{0}}\otimes m_{2n\underline{1}}),(e\otimes I_{C})\\
&\hspace{0,5cm}(m_{(2n+1)\underline{0}}\otimes m_{2n+1\underline{1}},...))\\
&=(e\otimes I_{C})((m_{0\underline{0}}\otimes m_{0\underline{1}}),.......,(m_{2n\underline{0}}\otimes m_{2n\underline{1}}),\\
&\hspace{0,5cm}(m_{2n+1\underline{0}}\otimes m_{2n+1\underline{1}},...))\\
&=(e\otimes I_{C})\circ \varrho^{M}(\sum_{n\geq0}m_{n}).
\end{align*}
Therefore, $e$ is an idempotent in $End^{C}(L)$.
Moreover, a $C$-comodule morphism $f:L\rightarrow L$ where $f(m_{n})=m_{n+1}$ in $End^{C}(L)$ and $u:=f-e$. Thus we have:
\begin{align*}
u(m_{2n})&=(f-e)(m_{2n})\\
&=f(m_{2n})-e(m_{2n})\\
&=m_{2n+1}-m_{2n}\\
\end{align*}
and
\begin{align*}
u(m_{2n+1})&=(f-e)(m_{2n+1})\\
&=f(m_{2n+1})-e(m_{2n+1})\\
&=m_{2n+2}-(m_{2n+2}-m_{2n})\\
&=m_{2n}.
\end{align*}
That is, for all $n\geq0$ we have
\begin{align*}
u\circ u (m_{2n})&=u(m_{2n+1}-m_{2n})\\
&=u(m_{2n+1})- u(m_{2n})\\
&=m_{2n}-(m_{2n+1}-m_{2n})\\
&=I_{L}(m_{2n})-u(m_{2n})\\
&=(I_{L}-u)(m_{2n})
\end{align*}
and
\begin{align*}
u\circ u (m_{2n+1})&=u(m_{2n})\\
&=m_{2n+1}-m_{2n}\\
&=I_{L}(m_{2n+1})-u(m_{2n+1})\\
&=I_{L}-u(m_{2n+1}).
\end{align*}
Therefore, $u\circ u=I_{L}-u$ if and only if $u(u+I_{L})=I_{L}$ for all $n\geq 0$. In particularly, $u$ is unit in $End^{C}(L)$, so $f=e+u$ is clean or $L$ is a clean $C$-comodule.
\end{proof}
Based on Lemma 3.3 and Remark 3.4 in \cite{cam06} we construct the set of all ordered pairs of $C$-subcomodules $M$ and define the invariant $f$-subcomodule. Let $f\in End_{R}(M)$. An $R$-submodule $N\subseteq M$ is said to be $f$-invariant if $f(N)\subseteq N$. Based on this concept we define the $f$-invariant subcomodules as follows.
\begin{definition} Let $M$ be a $C$-comodule and $f\in End^{C}(M)$. A $C$-subcomodule $W\subseteq M$ is said to be $f$-invariant if $f(W)\subseteq W$.
\end{definition}
Moreover, we define an essential monomorphism and co-Hopfian comodule.
\begin{definition}
Let $M$ be a $C$-comodule.
\begin{enumerate}
\item A monomorphism $f\in End^{C}(M)$ is called an \emph{essential monomorphism} if $Im(f)\subseteq^{e}M$.
\item $M$ is called \emph{a co-Hopfian} (resp. an essential co-Hopfian) if every monomorphism (resp. essential monomorphism) in $End^{C}(M)$ is onto.
\end{enumerate}
\end{definition}

Consider the dual algebra $C^{\ast}=Hom_{R}(C,R)$ by the convolution product (see Equation \ref{convo}). In \cite{wisbauer}, any $C$-comodule $M$ is a $C^{\ast}$-module. For any  $0\neq m$ in a right $C$-comodule $(M,\varrho^{M})$, we construct the set of $C^{\ast}\rightharpoonup m$ i.e.,
\begin{equation}\label{starset}
  C^{\ast}\rightharpoonup m =\{f\rightharpoonup m|f \in C^{\ast}\}
\end{equation}
where $f\rightharpoonup m=(I_{m}\otimes f)\circ \varrho^{M}(m)$. In \cite{wisbauer}, the category of $\mathbf{M}^{C}$ is a subcategory of $_{C^{\ast}}\mathbf{M}$ and it is become a full subcategory if and only $C$ satisfies the $\alpha$-condition. In this paper, we give a weaker condition than the $\alpha$-condition. For any $0\neq m\in M$, define a map
\begin{equation}\label{alphac}
\alpha_{M/C^{\ast}\rightharpoonup m}:M/C^{\ast}m\otimes_{R}C\rightarrow Hom_{R}(C^{\ast},M/C^{\ast}\rightharpoonup m), x\otimes c\mapsto [f\mapsto f(c)x].
\end{equation}
An $R$-coalgebra $C$ satisfies the $\alpha^{\ast}$-condition if the map $\alpha_{M/C^{\ast}\rightharpoonup m}$ is injective. Hence, if $C^{\ast}\rightharpoonup m$ is a $C$-pure $R$-submodule of $M$ or $C$ is a flat $R$-module, then the set $C^{\ast}\rightharpoonup m$ will always to be a $C$-subcomodule of $M$.

\begin{proposition}\label{essenmono} Let $C$ be a flat $R$-module and $C$ an $R$-coalgebra. Let $M$ be a quasi-continuous $C$-comodule, $W\subseteq^{e} M$, $f\in End^{C}(M)$ such that $f(W)\subseteq W$. If $f|_{W}-e$ is an essential monomorphism in $End^{C}(W)$ for some idempotent $e\in End^{C}(W)$, then there exists an idempotent $e'\in End^{C}(M)$ such that $e'|_{W}=e$ and $f-e'$ is an essential monomorphism in $End^{C}(M)$.
\end{proposition}
\begin{proof}
From assumption, $Im(f|_{W}-e)\subseteq^{e}W$. Since $M$ is a quasi-continuous $C$-comodule, we may extend $e$ as an idempotent on $M$ (see Lemma \ref{idemcom}). There is $e'\in End^{C}(M)$ such that $e'|_{W}=e$ and $f-e\in End^{C}(M)$. Now, we show that $f-e'$ is an essential monomorphism in $End^{C}(M)$.
\begin{enumerate}
\item Since $C$ satisfies the $\alpha^{\ast}$-condition, for any $x\in M, C^{\ast}\rightharpoonup x$ is a $C$-subcomodule of $M$.
\item Suppose that $W\subseteq^{e} M$. Clearly that $C^{\ast}\rightharpoonup x\cap W\neq 0$. It means there exists $0\neq g \in C^{\ast}$ such that $0\neq g\rightharpoonup x\in W$.
\item The monomorphism property of $f|_{W}-e$ implies that:
\begin{align*}
(f-e')(g\rightharpoonup x)&=f(g\rightharpoonup x)-e'(g\rightharpoonup x)\\
&=f|_{W}(g\rightharpoonup x)-e'|_{W}(g\rightharpoonup x)\\
&=f|_{W}(g\rightharpoonup x)-e(g\rightharpoonup x)\\
&=(f|_{W}-e)(g\rightharpoonup x).
\end{align*}
If $(f-e')(g\rightharpoonup x)=0$, then implies $(f|_{W}-e)(g\rightharpoonup x)=0$. By the injectivity of $(f|_{W}-e)$, $g\rightharpoonup x$ must be zero. Therefore, $f-e'$ is a monomorphism in $End^{C}(M)$.
\item Furthermore, we assumed that $Im(f|_{W}-e)\subseteq^{e}W$ and $W\subseteq^{e}M$. Thus, $Im(f|_{W}-e)$ is an essential $C$-subcomodule of $M$ (Lemma \ref{lemmaessen}).
\item Since $Im(f|_{W}-e)\subseteq Im(f-e')\subseteq M$ and $Im(f|_{W}-e)\subseteq^{e} M$, $Im(f-e')\subseteq^{e}M$. Consequently, $Im(f-e')$ is an essential monomorphism in $End^{C}(M)$ (Lemma \ref{lemmaessen}).
\end{enumerate}
\end{proof}

On the Proposition \ref{essenmono}, if we are assuming that $M$ is an essentially co-Hopfian comodule, then $f$ is onto. Thus, $f-e'$ is unit in $End^{C}(M)$. In particular, $f$ is clean.

\begin{remark}\label{remark1}
Let $M$ be a $C$-comodule and $f\in End^{C}(M)$. Let
\begin{center}
$\xi_{f}:=$ the set of all ordered pairs $(W,e)$
\end{center}
such that $W$ is an $f$-invariant $C$-subcomodule of $M$ and $e\in End^{C}(M)$ is an idempotent such that $f|_{W}-e$ is a unit in $End^{C}(W)$, or equivalently $f|_{W}$ is clean. As $(0,0)\in \xi_{f}$, the set $\xi_{f}$ is not empty. Let $(W_{1},e_{1}),(W_{2},e_{2})\in \xi_{f}$. We define a partial ordering by setting $(W_{1},e_{1})\leq (W_{2},e_{2})$ if and only if $W_{1}\subseteq W_{2}$ and $e_{2}|_{W_{1}}=e_{1}$. Thus, any totally ordered set $\{(W_{i},e_{i})|i\in I\}$ is bounded above by $(N,e)$ where $N=\cup_{i\in I}W_{i}$ and $e(x)=e_{i}(x)$ for all $x\in W_{i}$. By Zorn's Lemma, any $(W_{0},e_{0})\in \xi_{f}$ is bounded by a maximal element of $\xi_{f}$. It is clear that $f$ is clean in $End^{C}(M)$ if and only if there is element $(W,e)$ in $\xi_{f}$ with $W=M$.
\end{remark}

We use the remark to prove the cleanness of continuous comodules on the main Theorem.

To understand the structure of $\xi_{f}$, we observe the characteristic of the maximal elements on $\xi_{f}$.
\begin{lemma}\label{35} Let $C$ be a flat $R$-module and $C$ be an $R$-coalgebra with the $\alpha^{\ast}$-condition. If $f\in End^{C}(M), (W,e)$ is a maximal element in $\xi_{f}$, and $X$ is a $C$-subcomodule of $M$ where $X\cap W=0$, then we have the following statements:
\begin{description}
  \item[(A)] For any $x\in X$ if $f(x)\in W$, then $x=0$;
  \item[(B)] For any $m\in W\oplus X$ if $f(m)\in W$, then $m\in W$.
\end{description}
\end{lemma}
\begin{proof}
\begin{description}
  \item[(A)] Suppose that $x\in X$ and $f(x)\in W$. Put $w:= f(x)\in W$ and $X'=C^{\ast}\rightharpoonup x\subseteq X$. Therefore,
   \begin{enumerate}
   \item Since $C$ satisfies the $\alpha^{\ast}$-condition, $X=C^{\ast}\rightharpoonup x$ is a $C$-subcomodule of $M$.
   \item We want to prove that $W\oplus X'$ is an $f$-invariant as below:
   \begin{align*}
   f(W+X')&=f(W)+f(X')\\
   &\subseteq W\oplus f(C^{\ast} \rightharpoonup x) \text{( since $W$ is $f$-invariant)}\\
   &= W\oplus C^{\ast}\rightharpoonup f(x) \text{( by scalar muliplication of $C^{\ast}$)}\\
   &= W\oplus C^{\ast}\rightharpoonup w\\
   &\subseteq W (\text{( since $W$ is a $C^{\ast}$-module)})
   \end{align*}
   \end{enumerate}
   Consequently, $C$-comodule $W\oplus X'$ is $f$-invariant. Moreover, we also have some facts as below:
      \begin{enumerate}
      \item We can extends the idempotent endomorphism $e\in End^{C}(W)$ which is $(W,e)\in \xi_{f}$ to endomorphism of $X'$ by define $e(\alpha\rightharpoonup x)=\alpha\rightharpoonup e(x)=\alpha\rightharpoonup x$ for any $\alpha \rightharpoonup x\in X'$.
     \item For any $w'+\alpha \rightharpoonup x\in W\oplus X'$, we have:
      \begin{align*}
      e\circ e(w'+\alpha\rightharpoonup x)&=e\circ e(w')+e\circ e(\alpha \rightharpoonup x)\\
      &=e(w')+e(e(\alpha \rightharpoonup x))\\
      &=e(w')+e(\alpha \rightharpoonup e(x))\\
      &=e(w')+e(\alpha\rightharpoonup x)\\
      &=e(w')+\alpha\rightharpoonup e(x)\\
      &=e(w'+\alpha\rightharpoonup x).
      \end{align*}
      Hence, $e$ is an idempotent in $End^{C}(W\oplus X')$.
      \item Let $f-e:W\oplus X'\rightarrow W\oplus X'$ is a $C$-comodule morphism. We need to check that $(f-e)W=W$. Since $(W,e)\in \xi_{f}$, $(f-e)|_{W}$ is unit. That is, it is an automorphism and implies $(f-e)(W)\simeq W$. We will to continue our observation to prove $f-e\in End^{C}(W\oplus X')$ is also an automorphism as below:
      \begin{enumerate}
      \item We want to prove that $f-e:W\oplus X'\rightarrow W\oplus X'$ is onto. By using the equation $(f-e)(W)\simeq W$, for $w=f(x)\in W$ there is $w_{1}\in W$ such that $w=(f-e)(w_{1})$. Then we have,
          \begin{align*}
          (f-e)(w_{1}+ x)&=(f-e)(w_{1})-(f-e)( x)\\
          &=w-f(x)+e(x)\\
          &=f(x)-f(x)+e(x)\\
          &=x.
          \end{align*}
          Therefore, for any $w'+\alpha\rightharpoonup x\in W\oplus X'$
          \begin{align*}
          w'+\alpha\rightharpoonup x&=(f-e)(w_{2})+\alpha ((f-e)(w_{1}+ x)), \text{ (for some $w_{2}\in W)$}\\
          &=(f-e)(w_{2})+(f-e)(\alpha\rightharpoonup w_{1}+\alpha\rightharpoonup x)\\
          &=(f-e)(w_{2})+(f-e)(\alpha\rightharpoonup w_{1})+(f-e)(\alpha\rightharpoonup x)\\
          &=(f-e)(w_{2}+\alpha\rightharpoonup w_{1})+(f-e)(\alpha\rightharpoonup x).\\
          &=(f-e)((w_{2}+\alpha\rightharpoonup w_{1})+\alpha\rightharpoonup x)
          \end{align*}
          This means or any $w'+\alpha\rightharpoonup x\in W\oplus X'$ there exist $(w_{2}+\alpha\rightharpoonup w_{1})+\alpha\rightharpoonup x\in W\oplus X'$ such that $w'+\alpha\rightharpoonup x=(f-e)((w_{2}+\alpha\rightharpoonup w_{1})+\alpha\rightharpoonup x)$. In particularly, $W\oplus X'\in Im(f-e)$ or $f-e$ is onto.

      \item Next, suppose that $(f-e)(w_{1}+\alpha \rightharpoonup x)=0$ for some $w_{1}+\alpha \rightharpoonup x\in W\oplus X'$, then
       \begin{align*}
       0=(f-e)(w_{1}+\alpha \rightharpoonup x)&=(f-e)(w_{1})+ (f-e)(\alpha \rightharpoonup x)\\
       &=w_{2}+f(\alpha\rightharpoonup x)-e(\alpha \rightharpoonup x)\\
       &\text{ (for some $w_{2}=(f-e)(w_{1})\in W$)}\\
       &=(w_{2}+\alpha \rightharpoonup w)-\alpha \rightharpoonup x \text{ (since $w=f(x)$)}\\
       &=(w_{2}+\alpha \rightharpoonup w)-\alpha \rightharpoonup x
       \end{align*}
       Hence, we have that $\alpha \rightharpoonup x=w_{2}+\alpha \rightharpoonup w$ and implies $\alpha \rightharpoonup x\in W\cap X'$. Since $W\cap X=W\cap X'=\{0\}$, $\alpha \rightharpoonup x=0$. Then, $w_{2}+\alpha \rightharpoonup x=w_{1}$. Thus, $(f-e)(w_{2}+\alpha \rightharpoonup x)=(f-e)(w_{1})=0$. Moreover, since $(f-e)|_{W}$ is an automorphism implies that $w_{1}$ must be zero and $w_{1}+\alpha \rightharpoonup x=0$. In particularly, $f-e\in End^{C}(W+X')$ is a monomorphism.
       \end{enumerate}
       From this point the conclusion is $f-e\in End^{C}(W+X')$ is an automorphism.
       \item From the previous point for $C$-subcomodule $W\oplus X'$, we have that
       \begin{enumerate}
       \item $(f-e)|_{W\oplus X'}$ is an automorphism (unit) and $f$-invariant.
       \item If $e\in End^{C}(W)$, then $e$ is also an idempotent endomorphism of $W\oplus X'$.
       \end{enumerate}
       Therefore, $W\oplus X'\in \xi_{f}$. By the maximality of $(W,e)$, then $X'=0$ (since $W\subseteq W\oplus X'$.) Thus, when $X'=C^{\ast}\rightharpoonup x= 0$, then $x$ must be zero.
      \end{enumerate}
  \item[(B)] Now we prove that for any $m\in W\oplus X$ and $f(m)\in W$, it implies $m\in W$. \\
  Let $m=w+x\in W\oplus X$ such that $f(m)=f(w)+f(x)$. Since $W$ is $f$-invariant, $f(w)\in W$. Consequently, $f(x)=f(m)-f(w)\in W$. By using (A) we have $x=0$ and $m=w\in W$.
\end{description}
\end{proof}

We modify Theorem 3.7 in \cite{cam06} for the comodule case.
\begin{theorem}\label{37} Let $C$ be a flat $R$-module $C$ and be an $R$-coalgebra satisfying the $\alpha^{\ast}$-condition and $f\in End^{C}(M)$. Let $C$-comodule $M$ be either a semisimple $C$-comodule or a continuous comodule. No nonzero element of $M$ is annihilated by a left ideal of $C^{\ast}$. If $(W,e)\in \xi_{f}$, where $\xi_{f}$ is as in the Remark \ref{remark1}, then
\begin{description}
  \item[(A)] $(W,e)$ is a maximal element of $\xi_{f}$ if and only if $W=M$.
  \item[(B)] Given any $(W_{0},e_{0})\in \xi_{f}$, there exists a clean decomposition $f=e+u$ where $e$ is an idempotent of $End^{C}(M)$ extensions of
      $e_{0}$, and $u$ is a unit of $End^{C}(M)$. In particular, $M$ is a clean $C$-comodule.
  \item[(C)] Let $W_{1},W_{2}$ be $C$-subcomodules of $M$ with $W_{1}\cap W_{2}=0$ such that $f|_{W_{1}}$ and $1-f|_{W_{2}}$ are both automorphisms
      $C$-comodule. Then there exists a clean decomposition $f=u+e$  where $e$ is an idempotent element of $End^{C}(M)$ and $u$ is a unit of
      $End^{C}(M)$ restricted to zero on $W_{1}$ and to identity on $W_{2}$.
\end{description}
\end{theorem}
\begin{proof}
\begin{description}
  \item[(C)] Let $W_{1},W_{2}$ be $C$-subcomodules of $M$ with $W_{1}\cap W_{2}=0$ such that $f|_{W_{1}}$ and $1-f|_{W_{2}}$ are both automorphism
      $C$-comodules. Take $W_{0}=W_{1}\oplus W_{2}$ and $e_{0}$ as the projection of $W_{0}$ onto $W_{2}$ with kernel $W_{1}$, i.e.,
      \begin{equation*}
      e_{0}:W_{0}\rightarrow W_{2},w_{1}+w_{2}\mapsto w_{2}.
      \end{equation*}
      We will prove this point by following (B), Lemma \ref{NSCCleanCom1} and Proposition \ref{NSCCleanCom2}. By using definition of $e_{0}$, $W_{0}\in \xi_{f}$ since $f(W_{0})\subseteq W_{0}$ and $f|_{W_{0}}$ is clean. Thus, from (B) there is $f=e+u$ in $End^{C}(M)$ where $u$ is a unit and $e$ is an idempotent such that $e|_{W}=e_{0}$. For this point, we only need to check that $e_{0}|_{W_{1}}$ is a zero map and $u|_{W_{2}}$ is the identity $I_{W_{2}}$.
      \begin{enumerate}
      \item For any $w_{0}=w_{1}+w_{2}\in W_{0}$ with $w_{1}\in W_{1}$ and $w_{2}\in W_{2}, e_{0}\circ e_{0}(w_{0})=e_{0}(w_{0})=w_{2}$. Therefore, $e_{0}$ is an idempotent in $End^{C}(W_{0})$. For any $w_{1}\in W_{1}, w_{1}=w_{1}+0\in W_{0}$ and $e_{0}(w_{1}+0)=0$. Thus  $e_{0}|_{W_{1}}$ is a zero map.
      \item By \cite{wis91}, $W_{0}=W_{1}\oplus W_{2}=Ker(e_{0})\oplus Im(e_{0})$, such that $f(W_{1})\subseteq Ker(e_{0})=W_{1}$ and $(1-f)(W_{2})\subseteq Im(e_{0})=W_{2}$. From Proposition \ref{NSCCleanCom2} we have that $u|_{W_{0}}=f|_{W_{0}}-e_{0}\in End^{C}(W_{0})$ is a unit.
      \item For any $w_{2}=0+w_{2}\in W_{0}$, \\
      \begin{align*}
      u|_{W_{0}}(w_{2})&= f|_{W_{0}}-e_{0}(0+w_{2})\\
      &=f|_{W_{0}}(w_{2})+w_{2}\\
      &=f|_{W_{0}}(w_{2})-I_{W_{2}})(w_{2})\\
      &=-(1-f)_{W_{0}}(w_{2})\\
      &\subseteq Im(e_{0})\\
      &=W_{2}\\
      &=I_{W_{2}}(w_{2})
      \end{align*}
      It means, a unit $u$ in $End_{R}(M)$ restricted to identity in $W_{2}$.
      \end{enumerate}
\item[(B)] The proof is following from (A). From Remark \ref{remark1}, any $(W_{0},e_{0})\in \xi_{f}$ are bounded by a maximal element $(W,e)$ of $\xi_{f}$. From (A) if $(W,e)$ is the maximal element, then $W=M$. That is, there exists an idempotent $e\in End^{C}(M),e|_{W_{0}}=e_{0}$  which is $f-e$ is clean in $End^{C}(M)$.
\item[(A)] It is trivial that if $W=M$, then $(W,e)$ is a maximal element of $\xi_{f}$. Thus, we need only prove for the non-trivial "only if" part in (A), which will be presented in three steps below.

    Let $(W,e)$ be maximal.  We want to prove that $W=M$ if $M$ is either a semisimple $C$-comodule or $M$ is a continuous comodule and that no nonzero element of $M$ is annihilated by an essential left ideal of $C^{\ast}$.
\begin{enumerate}
  \item \textbf{Step 1.} We shall first prove that $W$ is a summand of $M$.
  \begin{enumerate}
  \item If $M$ is a semisimple $C$-comodule, there is nothing to prove since every subcomodule of $M$ is a direct summand \cite{wisbauer}.
  \item Let us assume that $M$ is a continuous comodule and that no nonzero element of $M$ is annihilated by an essential left ideal
  of $C^{\ast}$. Since $M$ is a continuous $C$-comodule, $M$ is $CS$ (satisfying $(M_{1})$), i.e., for every $C$-subcomodule $W$ of $M$ there exists a $C$-subcomodule $E\subseteq M$ where $E$ is a summand $M$ such that $W\subseteq^{e}E$. We will prove $W$ is a direct summand of $M$ by proving $W$ is (essentially) closed in $M$ (has no proper essential extensions subcomodule in $M$ such that $W$ is summand of $M$), i.e., $W=E$.
  \begin{enumerate}
  \item Let $E$ be a maximal essential extension $C$-subcomodule of $W$ in $M$ and $M=E\oplus X$ for some $C$-subcomodule $X$.
  \item For $y\in E$, let
      \begin{center}
        $I:=\{\alpha\in C^{\ast}| \alpha \rightharpoonup y\in W \}\subseteq C^{\ast}$.
      \end{center}
       For any $\alpha_{1},\alpha_{2}\in I$,
       \begin{enumerate}
       \item $(\alpha_{1}-\alpha_{2})\rightharpoonup y=(\alpha_{1}\rightharpoonup y)-(\alpha_{2}\rightharpoonup y)\in W$ (by scalar multiplication $"\rightharpoonup "$), then $\alpha_{1}-\alpha_{2}\in I$;
       \item and $(\alpha_{1}\ast\alpha_{2})\rightharpoonup y=\alpha_{1}\rightharpoonup(\alpha_{2}\rightharpoonup y)\in W$ (by $\ast$ as a scalar multiplication of $C^{\ast}$-module $M$, then $(\alpha_{1}\ast\alpha_{2})\in I$.
       \item Since $W$ is a $C^{\ast}$-module and $\alpha\rightharpoonup y\in W$, then for any $\beta\in C^{\ast},\beta\rightharpoonup(\alpha_{1}\rightharpoonup y)\in W$. Hence, $C^{\ast}I\subseteq I$. That is, $I$ is a left ideal of $C^{\ast}$.
       \end{enumerate}
  \item In this point we want to show $f(E)\subseteq E$. Since $M= E\oplus X$, for $y\in E, f(y)=z+x\in M$ for some $z\in E$ and $x\in X$. For any $\alpha_{1} \in I$,
      \begin{align*}
      \alpha_{1}\rightharpoonup f(y)&=\alpha_{1} \rightharpoonup (z+x)\\
      \Leftrightarrow \alpha_{1}\rightharpoonup f(y)&=\alpha_{1} \rightharpoonup z+\alpha_{1} \rightharpoonup x\\
      \Leftrightarrow f(\alpha_{1} \rightharpoonup y)&=\alpha_{1} \rightharpoonup z+\alpha_{1} \rightharpoonup x\\
      \Leftrightarrow -(\alpha_{1} \rightharpoonup x)&= \alpha_{1} \rightharpoonup z -(f(\alpha_{1} \rightharpoonup y))\\
      &\in E+W, \text{ (since $W$ is $f$-invariant and $\alpha_{1}\rightharpoonup y\in W$)}\\
      &\subseteq E, \text{ (since $W\subseteq E$)}
      \end{align*}
       This implies that $\alpha_{1}\rightharpoonup x\in E\cap X=\{0\}$ or $\alpha_{1}\rightharpoonup x=0$ for any $\alpha_{1}\in I$ and $I\subseteq Ann_{C^{\ast}}(x)\subseteq C^{\ast}$. 
        On the other hand  $C$-comodule $M$ has no nonzero element of $M$ that is annihilated by a left ideal of $C^{\ast}$ implies that if  $\alpha_{1}\rightharpoonup x =0$, then $x=0$. Thus, $f(y)=z+x=z$ for some $z\in E$ or $f(E)\subseteq E$.
  \item If the $C$-comodule $M$ is continuous and $E$ is a summand of $M$, then $E$ is also continuous. Here, let us collect our results, i.e., $W\subseteq^{e} E, E$ is continuous and $W$ is $f$-invariant, $(W,e)\in \xi_{f}$ such that $f|_{W}-e$ is a unit in $End^{C}(W)$ (essentially co-Hopfian). By Remark \ref{remark1}, there exists an idempotent $e'\in End^{C}(E)$ such that $e'|_{W}=e$ and $f|_{E}-e'$ is a unit in $End^{C}(E)$. That is, $(E,e')\in \xi_{f}$. By maximality of $(W,e)\in \xi_{f}$ implies that $W=E$ or $M=W\oplus X$.
  \end{enumerate}
  From here on, we shall assume that $M$ is a continuous comodule. Of course, semisimplicity will suffice.
  \end{enumerate}
  \item By Step 1, $M=W\oplus X$ and $W$ is $f$-invariant such that $f(M)=f(W)+f(X)\subseteq W+f(X)\subseteq M$. For any $x\in X, f(x+0)=f(x)$. Thus, $f:X\rightarrow f(X)$ is an isomorphism and $W\cap f(X)=0$. We are going to continue our work with two more steps, i.e., by assuming $M=W\oplus f(X)$ and $M\neq W\oplus f(X)$.\\

      \textbf{Step 2:} Let us assume that $M = W\oplus f(X)$. For the idempotent endomorphism $e\in End^{C}(W)$, take $A=Ker(e)$ and $B=Im(e)$. We already know that $f|_{W}-e$ is a unit, since $(W,e)\in \xi_{f}$. By using Lemma \ref{NSCCleanCom1}, $W=A\oplus B=C\oplus D$ where $f:A\rightarrow C$ and $1-f:B\rightarrow D$ are isomorphisms. That is,
      \begin{center}
        $M= (A\oplus B)\oplus X\simeq (C\oplus D)\oplus f(X)$.
      \end{center}
      Because of the isomorphic property of $f:A\rightarrow C$ and $f:X\rightarrow f(X)$, we obtain $f:A\oplus X\rightarrow C\oplus f(X)$ is also an isomorphism. By Lemma \ref{NSCCleanCom1}, we get $f-e$ is a unit in $End^{C}(M)$ or $(M,e')\in \xi_{f}$ with $e':M \rightarrow B$ is a projection with $Ker(e')=A\oplus X$. On the other hand, $(W,e)\leq (M,e')$. By using maximality of $(W,e)$, we have $W=M$.

   \textbf{Step 3:} Now we assume that $W\oplus f(X)\neq M$. Since $M$ is continuous and $X\simeq f(X)$, from $(CM_{1})$ and $(CM_{2})$ we have $W\oplus f(X)$ are summands of $M$. By using $(CM_{1})$ for every $0\neq v\in M$, $C$-subcomodule $C^{\ast}\rightharpoonup v$ is essential inside a (direct) summand of $M$. Since $W\oplus f(X)\neq M$, there exists a $0\neq v\in M$ such that
     \begin{equation}\label{cv}
      C^{\ast}\rightharpoonup v\cap (W\oplus f(X))=0.
      \end{equation}
  On the other hand, $M=W\oplus X$,
      \begin{equation}\label{1}
      f(M)=f(W)+f(X)\subseteq W\oplus f(X), \text{ (since $W$ is $f$-invariant)}
      \end{equation}
      \begin{enumerate}
      \item \textbf{Claim 1}. The sum $W+\sum_{i=0}^{n\in \mathbb{N}}C^{\ast}\rightharpoonup f^{i}(v)$ where $v\in V$ is direct.
  \begin{enumerate}
  \item Suppose that
       \begin{equation}\label{fw}
       0=w+\sum_{i=0}^{n\in \mathbb{N}} \alpha_{i}\rightharpoonup f^{i}(v)\in W
    \end{equation}
    where $w\in W$ and $\alpha_{i}\in C^{\ast}$ for any $i\in \mathbb{N}$. From Equation \ref{1} and Equation \ref{fw} we have:
 \begin{align*}
 \alpha_{0}\rightharpoonup f^{0}(v)=\alpha_{0}\rightharpoonup v\\
 &=-(w+\sum_{i=1}^{n\in \mathbb{N}} \alpha_{i}\rightharpoonup f^{i}(v))\\
 &\in W+f(M)\\
 \subseteq W\oplus f(X)
 \end{align*}
 Consequently, $\alpha_{0}\rightharpoonup v\in (C^{\ast}\rightharpoonup v)\cap (W\oplus f(X))$ and implies $\alpha_{0}\rightharpoonup v=0$ (see Equation \ref{cv}).
 This gives,
 \begin{align*}
& w+\sum_{i=1}^{n\in \mathbb{N}} f^{i}(\alpha_{i}\rightharpoonup v)=0\\
& \Leftrightarrow \sum_{i=1}^{n\in \mathbb{N}} f^{i}(\alpha_{i}\rightharpoonup  v)=-w, \text{ (for some $w\in W$)}\\
& \Leftrightarrow f(\alpha_{1}\rightharpoonup v+....f^{n-1}(\alpha_{n}\rightharpoonup v))\in W.
 \end{align*}
 Take
 \begin{align*}
  m&=\alpha_{1}\rightharpoonup v+....f^{n-1}(\alpha_{n}\rightharpoonup v)\\
  &\in C^{\ast}\rightharpoonup v+f(M)\\
  &\subseteq C^{\ast}\rightharpoonup v\oplus W\oplus f(X) \text{ (since $C^{\ast}\rightharpoonup v\cap (W\oplus f(X))={0}$)} \\
  &=(C^{\ast}\rightharpoonup v\oplus f(X))\oplus W.
  \end{align*}
   Consider $C^{\ast}\rightharpoonup v\oplus f(X)$ as a $C$-subcomodule of $M$ with
    \begin{center}
    $f(\alpha_{1}\rightharpoonup v+....f^{n-1}(\alpha_{n}\rightharpoonup v))\in W$.
\end{center}
Based on Lemma \ref{35} we have that $\alpha_{1}\rightharpoonup v+....f^{n-1}(\alpha_{n}\rightharpoonup v)\in W$. Consequently,
  \begin{align*}
  \alpha_{1}\rightharpoonup v+....f^{n-1}(\alpha_{n}\rightharpoonup v)&=w_{1} \text{ (for some $w_{1}\in W$)}\\
  \alpha_{1}\rightharpoonup v&=w_{1}-(f(\alpha_{2}\rightharpoonup v)+....f^{n-1}(\alpha_{n}\rightharpoonup v))\\
  &\subseteq W+f(M).
  \end{align*}
   Therefore, $\alpha_{1}\rightharpoonup v\in C^{\ast}\rightharpoonup v \cap W\oplus f(X)$ and implies $\alpha_{1}\rightharpoonup v=0$. Further repetition of this argument shows that $\alpha_{i}\rightharpoonup v=0$ for all $i$. Consequently, we will have the result that $w=0$ as well. On the other hand, we also find that
   \end{enumerate}
   \begin{center}
   $L:=\oplus_{i=0}^{n\in \mathbb{N}}C^{\ast}\rightharpoonup f^{i}(v)$
 \end{center}
 which is $L$ is nonzero since $v\neq 0$. Thus, $W\cap L =\{0\}$,
   \item \textbf{Claim 2} $f$ maps $f^{i}(C^{\ast}\rightharpoonup v)$ isomorphically onto $f^{i+1}(C^{\ast}\rightharpoonup v)$ for all $i\geq 0$.
   \begin{enumerate}
   \item Let $f(f^{i}(g\rightharpoonup v))=0\in W$ where $g\in C^{\ast}$. We consider $f^{i}(g\rightharpoonup v)\in f(M)\subseteq W\oplus f(X)$. Based on Lemma \ref{35}(B) and Claim 1, we have $f^{i}(g\rightharpoonup v)\in W\cap f^{i}(C^{\ast}\rightharpoonup v)=0$.
   \item Hence, $f|_{L}$ is the (forward) shift operator on $L$. By Lemma 3.1, we can find an idempotent $e'\in End_{R}(L)$ such that $f|_{L}-e'$ is a unit, which implies that $e\oplus e'$ is an idempotent endomorphism in $End^{C}(W\oplus L)$
   \item The $C$-subcomodule $W\oplus L$ is $f$-invariant, since
       \begin{align*}
       f|_{W\oplus L}&=f(W+L)\\
       &=f(W)+f(L)\\
       &=f(W)+f(\oplus_{i=0}^{n\in \mathbb{N}}C^{\ast}\rightharpoonup f^{i}(v))\\
       &=f(W)+\oplus_{i=0}^{n\in \mathbb{N}}C^{\ast}\rightharpoonup f^{i+1}(v)\\
       &\subseteq W\oplus L \text{ (since W is $f$-invariant)}.
       \end{align*}
   \item Moreover, $f|_{W\oplus L}-(e\oplus e')=(f|_{W}-e)\oplus (f|_{L}e')$ is a unit, since $f|_{W}-e$ and $f|_{L}-e'$ are unit.
   \end{enumerate}
   The explanations above give $(W\oplus L,e\oplus e')\in \xi_{f}$. This contradicts the maximality of $(W,e)$, which means that the case $W\oplus f(X)\neq M$ in Step 3 cannot really arise.
   \end{enumerate}
\end{enumerate}
\end{description}
\end{proof}

Based on Point (A) and (B) in Theorem \ref{37}, we reach an important conclusion as a consequence of the theorem. Since any $(W_{0},e_{0})\in \xi_{f}$ is always bounded by an element maximal in $\xi_{f}$, if $M$ is
\begin{enumerate}
\item a semisimple $C$-comodule or
\item a continuous $C$-comodule with no nonzero element of $M$ that is annihilated by an essential left ideal $C^{\ast}m$
of the dual algebra $C^{\ast}$,
\end{enumerate}
then $M$ is a clean $C$-comodule.

\end{document}